%% file: alternating-sign-coxeter.tex
\documentclass[12pt]{amsart}

\usepackage{amsmath}
\usepackage{amssymb}
\usepackage{graphicx}
\usepackage{color}
\usepackage{url}
\usepackage{epstopdf}
\input xy
\xyoption{all}

\newtheorem{thm}{Theorem}

\newtheorem{cor}[thm]{Corollary}
\newtheorem{lem}[thm]{Lemma}

\newtheorem{conj}[thm]{Conjecture}
\newtheorem{prop}[thm]{Proposition}
\newtheorem{rem}[thm]{Remark}
\newtheorem{ex}[thm]{Example}

\newcommand{\R}{\mathbf{R}} 
 
\newcommand{\Z}{\mathbf{Z}} 

\begin{document}
\title{On Coxeter mapping classes and fibered alternating links}
\author{Eriko Hironaka}
\author{Livio Liechti}
\thanks{\small 
\noindent 
The first author was partially supported by a grant from the Simons Foundation (\#209171). The second author was supported by the Swiss National Science Foundation (\#159208).}
\address{Eriko Hironaka, Department of Mathematics, Florida State University, Tallahassee, FL 32301-4510, USA}
\email{hironaka@math.fsu.edu}
\address{Livio Liechti, Mathematisches Institut der Universit\"at Bern, Sidlerstrasse 5, 3012 Bern, Schweiz}
\email{livio.liechti@math.unibe.ch}

\pagestyle{plain}

\begin{abstract}  Alternating-sign Hopf plumbing along a tree yields fibered alternating links whose homological monodromy is, up to a sign, conjugate to some alternating-sign Coxeter transformation.
Exploiting this tie, we obtain results about the location of zeros of the Alexander polynomial of the fibered link complement
 implying
a strong case of Hoste's conjecture, the trapezoidal conjecture, bi-orderability of the link group, 
 and a sharp lower bound for the homological dilatation of the monodromy of the fibration.
The results extend
to more general hyperbolic fibered 3-manifolds associated to alternating-sign Coxeter graphs.\end{abstract} 
\maketitle

\section{Introduction}
In this paper, we study mapping classes defined by  bipartite Coxeter graphs with  sign-labels on the vertices
determined by the bipartite structure.
 If the graph is connected and has at least two vertices, then these {\it alternating-sign Coxeter 
mapping classes} are pseudo-Anosov, and if the Coxeter graph is a tree the associated mapping class is
the monodromy of an alternating fibered knot or link, which we call an {\it (alternating) Coxeter link}.

There has long been interest in the location of roots of Alexander polynomials for alternating links.  Murasugi
showed that  the coefficients of the polynomials have alternating signs,
and hence no real root can be negative \cite{Mur}.  Hoste conjectured that the real part of all zeros 
must be bounded from below by $-1$.  This and related conjectures were settled for some classes of alternating
links in \cite{LM, KP, Sto, HM}.  

Using properties of alternating-sign Coxeter transformations, we give a simple proof that the roots of the Alexander
polynomials for alternating Coxeter links are real and positive.  By a result of Perron and Rolfsen \cite{PR},
this implies that the fundamental group of the complement of an alternating Coxeter link is
bi-orderable.   Applying an interlacing property for alternating-sign
Coxeter graphs, we show that the homological dilatations are monotone under graph inclusion.  Thus the minimum
homological dilatation achieved by an alternating Coxeter link is $\frac{3 + \sqrt{5}}{2}$, the square of the golden ratio.
Similar properties hold for the Alexander polynomial of the mapping torus of alternating-sign Coxeter mapping classes.

\begin{rem}{\em  In \cite{HM} Hirasawa and Murasugi similarly study the roots of Alexander polynomials for
quasi-rational knots and links, which include the Coxeter
links discussed in this paper, and they also prove stability and interlacing properties of the Alexander polynomial
for these examples.   By applying the constructs of Coxeter graphs and Coxeter transformations in this paper, we
simplify their proofs in this context, and extend the results to more general mapping classes
and mapping tori associated to alternating-sign Coxeter graphs.
}
\end{rem}

\subsection{Alexander polynomials of alternating knots and links}
The Alexander polynomial  $\Delta(t) \in \Z[t]$ is an invariant of a finitely presented group with a prescribed homomorphism onto $\Z$.
Given a knot or link $K$ in $S^3$, each oriented Seifert surface $S$ defines a surjective homomorphism of $\pi_1(S^3 \setminus K)$ to $\Z$ by algebraic intersection of closed paths with $S$. Denote by $\Delta_{S}(t)$ the associated Alexander polynomial.  
If $M = S^3 \setminus K$ 
is fibered over the circle with fiber $S$ and monodromy $\phi$, then $\Delta_{S}(t)$ is the characteristic polynomial of
the homological monodromy $\phi_{\mbox{hom}} : H_1(S;\R) \rightarrow H_1(S;\R)$ (this can be deduced from either the
Fox calculus or the Seifert algorithm for finding $\Delta_S(t)$, see e.g.~\cite{Rol}).  Given any mapping 
class $\phi$
on a surface $S$,  write $\Delta_{S,\phi}(t)$ for the characteristic polynomial of the homological monodromy.
It follows that if $K$ is a fibered link with monodromy $(S,\phi)$, and $\Delta_K(t)$ is the Alexander polynomial of $K$, we have
$$
\Delta_K(t) = \Delta_S(t) = \Delta_{S,\phi}(t).
$$

There are few restrictions on the Alexander polynomial:  any monic, reciprocal polynomial 
can be realized as $\Delta_{S,\phi}(t)$ up to multiples of $t$ and  $(t-1)$, where $(S,\phi)$ is the monodromy of some fibered link~\cite{Kan}.  The story is different when we confine ourselves to
{\it alternating knots and links}: those that admit
a planar projection such that over and under crossings are alternating.  Murasugi showed in \cite{Mur} that if $S$
is the Seifert surface defined by an alternating planar projection, then
$\Delta(-t)$ has degree $2g$, and the coefficients for the powers $t^{k}$ are all strictly positive or strictly negative for
$0 \leq k \leq 2g$.   This implies, for example, that any real root of $\Delta(t)$ must be positive.

In 2002, Hoste conjectured the following:

\begin{conj}[Hoste]  For alternating knots, the real part of any zero of the Alexander polynomial 
is strictly greater than~$-1$. 
\end{conj}

\noindent
A lower bound on the real part of roots of $\Delta(t)$ was found by Lyubich and Murasugi~\cite{LM} 
for two-bridge links.  The results were later improved by Koseleff and Pecker~\cite{KP},  and Stoimenow~\cite{Sto}.
Hirasawa and Murasugi in \cite{HM} showed that for a large class of alternating links, the roots of the Alexander
polynomial are real and positive, a property of integer polynomials known as {\it real stability}.

Our first result is the following.

\begin{thm}\label{realroots-thm} If $(S,\phi)$ is an alternating-sign Coxeter mapping class, then 
$\Delta_{(S,\phi)}(t)$ has real stability.   In particular, the  Alexander polynomial of an alternating Coxeter link has real stability.
\end{thm}

Fox's \emph{trapezoidal conjecture} concerns the coefficients of Alexander polynomials of alternating knots.

\begin{conj}[\cite{Fox}]
Let $\Delta(t) = a_{2g}t^{2g} + \dots + a_0$ be the Alexander polynomial of an alternating knot. 
Then there exists an integer $k$ satisfying $0\le k\le g$ such that $$|a_0|< \dots <|a_k| = \dots = |a_{2g-k}| > \dots > |a_{2g}|.$$
\end{conj}

\noindent
The trapezoidal conjecture has been verified for several classes of alternating knots, 
e.g.\ for algebraic alternating knots by Murasugi~\cite{Mur2} and alternating knots of genus two by Ozsv\'{a}th and Szab\'{o}~\cite{OZ} and Jong~\cite{Jo}.

Real stability implies the trapezoidal property for  integer polynomials.
The coefficient sequence of a polynomial  $a_{2g}t^{2g} + \dots + a_0 \in \R[t]$ with only positive real roots
is strictly \emph{log-concave}, 
i.e. $$a_i^2 > a_{i-1}a_{i+1}$$ holds for all $i=2, \dots ,2g-1$, see e.g.~\cite{Wag}.
 Thus, the trapezoidal property  of Alexander polynomials 
 of alternating Coxeter links 
 follows from  Theorem~\ref{realroots-thm} (cf.~\cite{HM}).  More generally, we have the following.

\begin{cor}\label{trapezoidal-thm} If $(S,\phi)$ is an alternating-sign Coxeter mapping class, then $\Delta_{(S,\phi)}(t)$ 
is trapezoidal.  In particular, alternating-sign Coxeter links have trapezoidal Alexander polynomials.
\end{cor}

\subsection{Bi-orderable groups}

A second application of Theorem~\ref{realroots-thm} is the bi-orderability of knot groups and fundamental
groups of 3-manifolds.

A group $G$ is \emph{bi-orderable} if it admits a total order $<$ on $G$ that is compatible with the group operation, that is
$$
a \leq b \ \mbox{and}\ c \leq d \qquad \mbox{implies} \qquad ac  \leq bd.
$$
Perron and Rolfsen showed that if all the eigenvalues of the homological action of a surface homeomorphism $\phi$ are real and positive, 
then the fundamental group of its mapping torus is bi-orderable~\cite{PR, PR2}.  
Thus, Theorem~\ref{realroots-thm} has this immediate consequence.

\begin{cor}\label{biorderable-thm}
The mapping torus of an alternating-sign Coxeter mapping class has bi-orderable fundamental group.
\end{cor}


\subsection{Dilatations of mapping classes}
A {\it mapping class} on an oriented compact surface $S$ of finite type is a self-homeomorphism up to isotopy relative
to the boundary.  The {\it homological dilatation} $\lambda_{\mbox{hom}}$
of a mapping class $\phi$ is the largest eigenvalue (in modulus) of the
characteristic polynomial of the action of $\phi$ on first homology.
By the Nielsen-Thurston classification theorem, mapping classes fall into three types: those that
are periodic, non-periodic but preserving the isotopy class of a simple closed multi-curve, and {\it pseudo-Anosov}.
The third type is the most general, and has the property that for some pair of transverse measured singular foliations
$(\mathcal F^\pm, \nu^\pm)$, the mapping class stretches the measure $\nu^-$ by $\lambda$ and $\nu^+$ by
$\lambda^{-1}$ for some $\lambda > 1$.  
The constant  $\lambda_{\mbox{geo}} = \lambda$ is the {\it (geometric) dilatation} of the mapping class.
The homological and geometric dilatations are related as follows
$$
\lambda_{\mbox{hom}}(\phi) \leq \lambda_{\mbox{geo}}(\phi),
$$
with equality if and only if $\phi$ is {\it orientable}, i.e., its
invariant foliations $\mathcal F^{\pm}$ are orientable (see, e.g.~\cite{FM}).  

The mapping torus of a mapping class $(S,\phi)$  is the 3-dimensional manifold 
 $$
 M = M_{(S,\phi)} = S \times [0,1]/(x,1) \sim (\phi(x),0).
 $$
By a theorem of W. Thurston, this manifold admits a hyperbolic structure
 if and only if $\phi$ is pseudo-Anosov \cite{Thu}.   The associated fibration $M \rightarrow S^1$ defines a
surjective homomorphism $\pi_1(M) \rightarrow \Z$ and a corresponding Alexander polynomial  $\Delta_{(S,\phi)}(t)$.

We show that the dilatation of alternating-sign Coxeter mapping classes is monotonic with respect
to graph inclusion.  Thus the minimum dilatation for alternating-sign Coxeter mapping classes
is achieved by the alternating-sign $A_2$ graph, which in turn is geometrically realized by the
figure eight knot.

\begin{thm}
\label{goldenmean-thm} 
The minimum homological and geometric dilatation of alternating-sign Coxeter mapping classes is 
the square of the golden ratio $\frac{3 + \sqrt{5}}{2}$, and is geometrically realized as the monodromy of the figure eight knot.
\end{thm}

\begin{rem} \label{Lehmer-rem}  {\em By a result of McMullen \cite{McM}
the spectral radius of the classical Coxeter transformations is minimized by
the $E_{10}$ Coxeter graph, also known as the $(2,3,7)$ star-like graph \cite{MRS}. The associated Coxeter link is
 the $(-2,3,7)$-pretzel link \cite{Hi1} and the dilatation of its monodromy is the conjectural smallest Salem number,
 known as Lehmer's number \cite{Le}, which is smaller than the square of the golden ratio.}
\end{rem}

\begin{rem} {\em By contrast to Theorem~\ref{goldenmean-thm}, when dropping the assumption of alternating signs,
it is possible to find mixed-sign Coxeter graphs whose associated mapping classes have dilatation arbitrarily close to 1~(see \cite{Hi2}).}
\end{rem}

\subsection{Organization} In Section~\ref{Coxeter-sec} we recall some definitions and properties of classical Coxeter systems
and generalize them to mixed-sign Coxeter systems.   The analog of Alexander polynomials for Coxeter systems is
the Coxeter polynomial, the characteristic polynomial of the Coxeter transformation. 
For bipartite alternating-sign Coxeter systems, we prove real stability 
for the Coxeter polynomial and the interlacing property.
   Section~\ref{geometry-sec} discusses geometric realizations of alternating-sign
Coxeter systems and
contains proofs of Theorems~\ref{realroots-thm} and~\ref{goldenmean-thm}.
\medskip

\noindent
{\bf Acknowledgements.} The authors are grateful to N.\ A'Campo, S.\ Baader, K.\ Murasugi and to the anonymous referee
for their helpful comments and suggestions.

\section{Bipartite Coxeter graphs}\label{Coxeter-sec}

A \emph{mixed-sign Coxeter graph} is a pair $(\Gamma, \mathfrak{s})$, 
where $\Gamma$ is a finite connected graph without self- or double edges and $\mathfrak{s}$ is an assignment of a sign $+$ or $-$ to every vertex $v_i$ of $\Gamma$. 
Let $\R^{V_\Gamma}$ be the vector space of $\R$-labelings 
of the vertices of $\Gamma$.  For $v \in V_\Gamma$, let $[v]$ be the corresponding element of $\R^{V_\Gamma}$ giving the label
$1$ on $v$ and $0$ on all other vertices of $\Gamma$.
The real vector space $\R^{V_{\Gamma}}$ is equipped with a symmetric bilinear form $B$, 
given by $B([v_i], [v_i]) = -2\cdot\mathfrak{s}(v_i)$ and otherwise $B([v_i], [v_j]) = a_{ij}$, where $A = (a_{ij})$ is the adjacency matrix of $\Gamma$.
To every vertex $v_i$, we associate a reflexion $s_i$ about the hyperplane of $\R^{V_{\Gamma}}$ perpendicular to $[v_i]$, given by the formula $$s_i([v_j]) = [v_j] - 2\frac{B([v_i],[v_j])}{B([v_i],[v_i])}[v_i].$$
The \emph{Coxeter transformation} is the product $C = s_1 \cdots s_n$ of all these reflections. 
For trees, this product does not depend, up to conjugation, on the order of multiplication~\cite{Ste}, but in general it does. 
For bipartite Coxeter graphs $\Gamma$, however, there is a distinguished conjugacy class, the \emph{bipartite Coxeter transformation} $C_{+-}$ given by $C_{+-} = C_+  C_-$, 
where $C_+$ is any product of all the reflections corresponding to vertices in one part of the partition and $C_-$ is any product of all the reflections corresponding to vertices in the other part.
This is well-defined since all the reflections corresponding to vertices in one part of the partition commute pairwise. 

If all signs $\mathfrak{s}$ of a bipartite Coxeter graph are positive, theorems of A'Campo and McMullen state that the eigenvalues of the bipartite Coxeter transformation are on the unit circle or positive real 
and that the spectral radius is monotonic with respect to graph inclusion~\cite{AC1, McM}.
We now prove analogs of these theorems for \emph{alternating-sign Coxeter graphs}, the case where the bipartition of the graph $\Gamma$ is actually given by the signs $\mathfrak{s}$.

\begin{prop} 
\label{bicolour-prop}
Let $(\Gamma, \mathfrak{s})$ be an alternating-sign Coxeter graph. 
Then the eigenvalues of the bipartite Coxeter transformation $C_{+-}$ are real and strictly negative.
\end{prop}

\begin{proof}
Let $(\Gamma, \mathfrak{s})$ be an alternating-sign Coxeter graph. 
Number the vertices of $\Gamma$ starting with all the positive ones, and then proceeding to the negative ones. 
With this vertex numbering, the adjacency matrix $A = A(\Gamma)$ of $\Gamma$ becomes a $2\times2$-block matrix with zero blocks on the diagonal and blocks $X$ and $X^\top$ in the upper right and lower left, respectively. Using the above formula for the $s_i$, we have that the products $C_+$ and $C_-$ corresponding to the partition are given by
$$C_{+} = \begin{pmatrix}
  -I & X\\
 0 & I
\end{pmatrix} \text{, }
C_{-} = \begin{pmatrix}
 I & 0\\
 -X^\top & -I
\end{pmatrix}.$$
 Multiplication of $C_+$ and $C_-$ shows that the bipartite Coxeter transformation $C_{+-}=C_{+}C_{-}$ is symmetric. Therefore, $C_{+-}$ has only real eigenvalues.
 It is left to show that there are no positive eigenvalues. 
Note that $(C_{+} + C_{-})^2 = -A(\Gamma)^2$. Furthermore, by expanding we obtain 
\begin{align*}
 (C_{+} + C_{-})^2 &= 2I + C_{+-} + C_{+-}^{-1}
\end{align*}
and thus, for any eigenvalue $\lambda \in \R$ of $C_{+-}$, we have $$2+\lambda + \lambda^{-1} = -\alpha^2,$$ where $\alpha$ is some eigenvalue of the adjacency matrix $A(\Gamma)$.
It follows that $2+\lambda + \lambda^{-1}$ is a non-positive real number, since $\alpha$ is a real number. 
In particular, every eigenvalue $\lambda$ of the alternating-sign Coxeter transformation $C_{+-}$ is strictly negative.
\end{proof}

\subsection{Interlacing property}
Let $\Gamma$ and $\Gamma'$ be alternating-sign Coxeter graphs so that $\Gamma$ is a subgraph
of $\Gamma'$.  We say that
$\Gamma'$ is obtained from $\Gamma$ by a {\it vertex extension} if the vertex set of
$\Gamma'$ contains one more element $w$ than the vertex set of $\Gamma$, and the edges of $\Gamma$
are precisely the edges of $\Gamma'$ that do not have $w$ as an endpoint.



\begin{prop}
\label{interlacing-prop}
Let $(\Gamma, \mathfrak{s})$ and $(\Gamma', \mathfrak{s}')$ be two alternating-sign Coxeter graphs. 
If $\Gamma'$ is a vertex extension of $\Gamma$, then the eigenvalues of the
bipartite Coxeter transformations $C_{+-}$ and $C_{+-}'$ are interlaced,
i.e., if 
$\alpha_1 \leq \dots \leq \alpha_s$ are the eigenvalues of $C_{+-}(\Gamma)$, and $\beta_1 \leq \dots \leq\beta_{s+1}$
are the eigenvalues of $C_{+-}(\Gamma')$, then
$$
\beta_1 \leq \alpha_1 \leq  \beta_2 \leq \cdots \leq \alpha_s \leq \beta_{s+1}.
$$
\end{prop}

\begin{proof}
Let $(\Gamma, \mathfrak{s})$ be an alternating-sign Coxeter graph with bipartite Coxeter transformation $C_{+-}$. 
From the proof of Proposition~\ref{bicolour-prop}, we recall that the eigenvalues of $C_{+-}$ are in one-to-one correspondence with the eigenvalues of the adjacency matrix $A(\Gamma)$.
More precisely, the correspondence is given by $$-\alpha^2 = 2 + \lambda + \lambda^{-1},$$
where $\lambda$ and $\alpha$ are eigenvalues of $C_{+-}$ and $A(\Gamma)$, respectively. 
Since $\Gamma$ is bipartite, the eigenvalues of $A(\Gamma)$ are symmetric with respect to the origin~\cite{BrHa}. 
Furthermore, since $\text{max}(\vert \lambda \vert, \vert \lambda \vert^{-1})$ is monotonically increasing with respect to $\alpha^2$, 
there exists a monotonic transformation of $\R$ taking the eigenvalues of $A(\Gamma)$ to the eigenvalues of $C_{+-}$.
Now let $(\Gamma', \mathfrak{s}')$ be an alternating-sign Coxeter graph with bipartite Coxeter transformation $C_{+-}'$ such that $\Gamma'$ is a vertex-extension of $\Gamma$.
Then the eigenvalues of $A(\Gamma)$ and $A(\Gamma')$ are interlaced~\cite{BrHa} and therefore so are the eigenvalues of $C_{+-}$ and $C_{+-}'$.
\end{proof}

\begin{prop} \label{minimum-prop} The minimum spectral radius for an
alternating-sign Coxeter transformation  is realized by the alternating-sign $A_2$ Coxeter graph, and the
spectral radius is the square of the golden mean.
\end{prop}

\begin{proof}
Noting that every non-trivial alternating-sign Coxeter graph
is a (perhaps multiple) vertex extension of the alternating-sign $A_2$ graph, the statement follows from Proposition~\ref{interlacing-prop}.
\end{proof}

\begin{rem}{\em
If a bipartite graph $\Gamma$ is a subgraph of another bipartite graph $\Gamma'$ with one more vertex but $\Gamma'$ is not a vertex-extension of $\Gamma$,
then the eigenvalues of the corresponding adjacency matrices need not be interlaced. 
\begin{figure}[h]
  \def\svgwidth{300pt}
  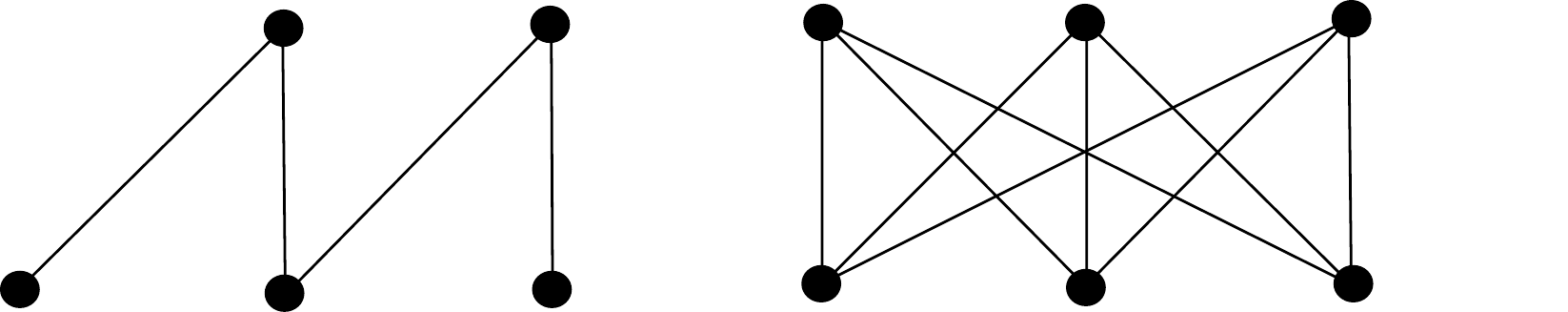
  \caption{}
  \label{non-interlacing}
\end{figure}
Choosing $\Gamma$ and $\Gamma'$ as in Figure~\ref{non-interlacing},
the eigenvalues of the adjacency matrix of $\Gamma$ are given by $\{-\sqrt{3}, -1, 0, 1, \sqrt{3}\}$ and the eigenvalues of the adjacency matrix of $\Gamma'$ are given by $\{-3, 0, 0, 0, 0, 3\}$.
In particular, these eigenvalues are not interlaced. However, focusing on the largest eigenvalue, it is still true that the spectral radius is monotonic under graph inclusion.
}
\end{rem}

\begin{prop}
\label{monoton-prop}
Let $(\Gamma, \mathfrak{s})$ and $(\Gamma', \mathfrak{s}')$ be two alternating-sign Coxeter graphs. 
If $\Gamma$ is a subgraph of $\Gamma'$, then the spectral radius of $C_{+-}$ is less than or equal to the spectral radius 
of $C_{+-}'$.
\end{prop}

\begin{proof}
The proof is basically the same as the proof of Proposition~\ref{interlacing-prop}. However, instead of interlacing (which does not necessarily apply in the case of non-induced subgraphs), 
we use Perron-Frobenius theory and the fact that $A(\Gamma)$ is dominated by a submatrix of $A(\Gamma')$.
\end{proof}

\begin{rem}{\em
General Coxeter graphs are defined with arbitrary edge weights $m_{ij} \ge 3$. The corresponding entries $a_{ij}$ of the adjacency matrix are then defined to be $a_{ij} = 2\cdot\text{cos}(2\pi/m_{ij})$.
Although we formulated Propositions~\ref{bicolour-prop} and~\ref{monoton-prop} for constant edge-weights $m_{ij} = 3$,  they also hold in this generalized context. 
Proposition~\ref{bicolour-prop} holds without change of wording.
For Proposition~\ref{monoton-prop}, we must add the assumption that when $\Gamma$ is a subgraph of $\Gamma'$, then every edge-weight of $\Gamma$ is less than or equal to the edge-weight of $\Gamma'$.  
}
\end{rem}

\section{Geometric realization}\label{geometry-sec}

In this section, we associate fibered alternating links and more general mapping tori to alternating-sign Coxeter graphs $(\Gamma,\mathfrak{s})$.

\subsection{Mapping classes from mixed-sign Coxeter systems}
Mixed-sign Coxeter systems, defined by Coxeter graphs with ordered, signed vertices, are useful
for building examples of mapping classes.

 As in the {\it classical} (or {\it positive-sign}) case, a mixed-sign Coxeter graph with $n$ vertices defines a
 subgroup of the general linear group
 $\mbox{GL}(n,\mathbb R)$ generated by reflections.  In the classical case, the reflections
  preserve an associated symmetric bilinear form
$2I - A$, where $A$ is the adjacency matrix of the Coxeter graph.
For a mixed-sign Coxeter system the bilinear form is given  by
$2I_{\mathfrak s} - A$, where $I_{\mathfrak s}$ is a diagonal matrix with $\pm 1$ entries on the diagonal
depending on the signs $\mathfrak s$ assigned to vertices of the Coxeter graph.
For mixed-sign Coxeter graphs, just as for classical ones, one can explicitly construct mapping classes whose homological 
monodromy is  conjugate to the Coxeter transformation up to sign \cite{Hi1, Hi2, Le, Thu}.  

Classical bipartite Coxeter systems have been shown to have many useful properties. 
A'Campo showed that all 
eigenvalues of the Coxeter transformation are real or lie on the unit circle.  This condition is sometimes
called {\it bi-stability}  \cite{HM}.  Since the traces of the eigenvalues over the reals are related to the
eigenvalues of the adjacency matrix of the Coxeter graph, the eigenvalues satisfy an interlacing theorem.
McMullen used this to prove monotonicity of the spectral radius of Coxeter transformations with respect
to graph inclusion, and found a sharp lower bound for the gap between 1 and the next smallest
spectral radius of Coxeter transformations  \cite{McM}.
It follows, in particular, that the classical Coxeter mapping classes associated to bipartite classical Coxeter
graphs that are not spherical or affine have dilatation bounded from below by Lehmer's number, which is approximately 1.17628 \cite{Le}.

\begin{rem}\label{positivity-rem} {\em  By contrast to Theorem 3, A'Campo showed that for
any classical bipartite Coxeter graph that is not spherical or affine,
the roots of the corresponding Coxeter polynomials are either on the
unit circle or positive real, with at least one root greater than 1~\cite{AC1}. If Hoste's
conjecture is true, this gives a homological proof of the fact 
that the knots associated to classical bipartite Coxeter graphs 
that are not spherical or affine can never be alternating. 
This can also be proved independently: such a knot is positive, i.e.\ it has a diagram with only positive crossings.
for the signature $\vert\sigma\vert$ and genus $g$, we have $\vert\sigma\vert < 2g$, since 
$2g$ equals the number of vertices and $\vert\sigma\vert$ equals the signature 
of the bilinear form $2I - A$. But for knots which are both positive and alternating, 
$\vert\sigma\vert = 2g$ holds, e.g.\ by properties of Rasmussen's $s$-invariant~\cite{Ra}.
}
\end{rem}

Let $\mathcal L$ be
an arrangement of line segments in the plane whose intersection graph equals $\Gamma$.   That is, to each vertex $v$ of $\Gamma$
there is an associated line segment  $\ell_v$ in $\mathcal L$, and two line segments in $\mathcal L$ intersect if the corresponding vertices are connected by
an edge of $\Gamma$.
A {\it planar realization of} $\Gamma$ is an embedding of $\mathcal  L$ in $\R^2$ 
with coordinate
axes $x$ and $y$, so that  if  $\mathfrak{s}(v) = 1$, then $\ell_v$
is parallel to the $y$-axis, and if $\mathfrak{s}(v) = -1$, then $\ell_v$ is parallel to the $x$-axis.

If $\Gamma$ has a planar realization, then we thicken the $\ell_v$ into rectangular strips $\ell_v \times [-1,1]$ (resp., $[-1,1] \times \ell_v$),
so that each segment $\ell_v$ is identified with $\ell_v \times \{0\}$ (resp., $\{0\} \times \ell_v$). 
If $v$ and $w$ are adjacent on $\Gamma$ then the rectangular strips
 $\ell_v$ and $\ell_w$ are glued together at right angles as in Figure~\ref{crossing-fig}. 
   \begin{figure}[htbp] 
    \centering
    \includegraphics[width=2in]{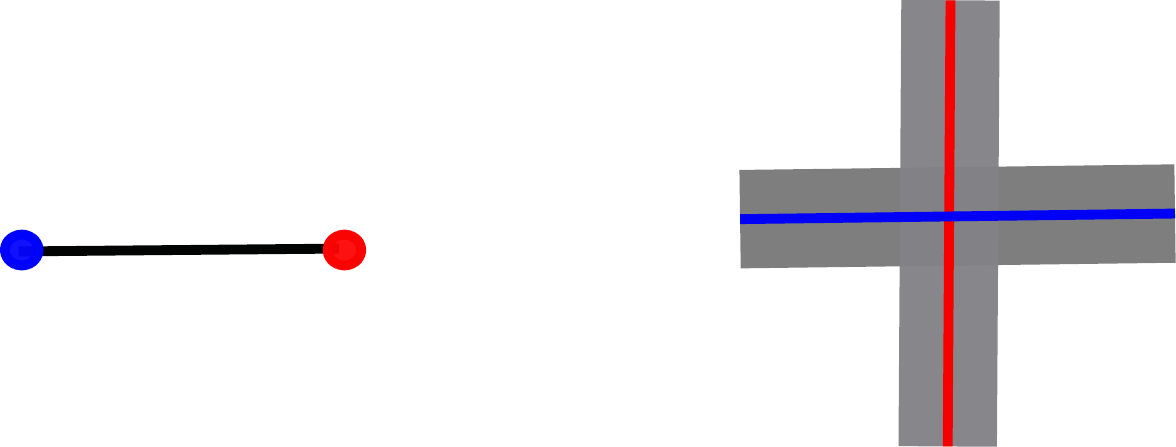} 
    \caption{}
    \label{crossing-fig}
 \end{figure}
 The thickenings and gluings can be made so that 
 all rectangular strips in each bipartite partition are parallel to one another.
 \begin{figure}[htbp] 
   \centering
   \includegraphics[width=4in]{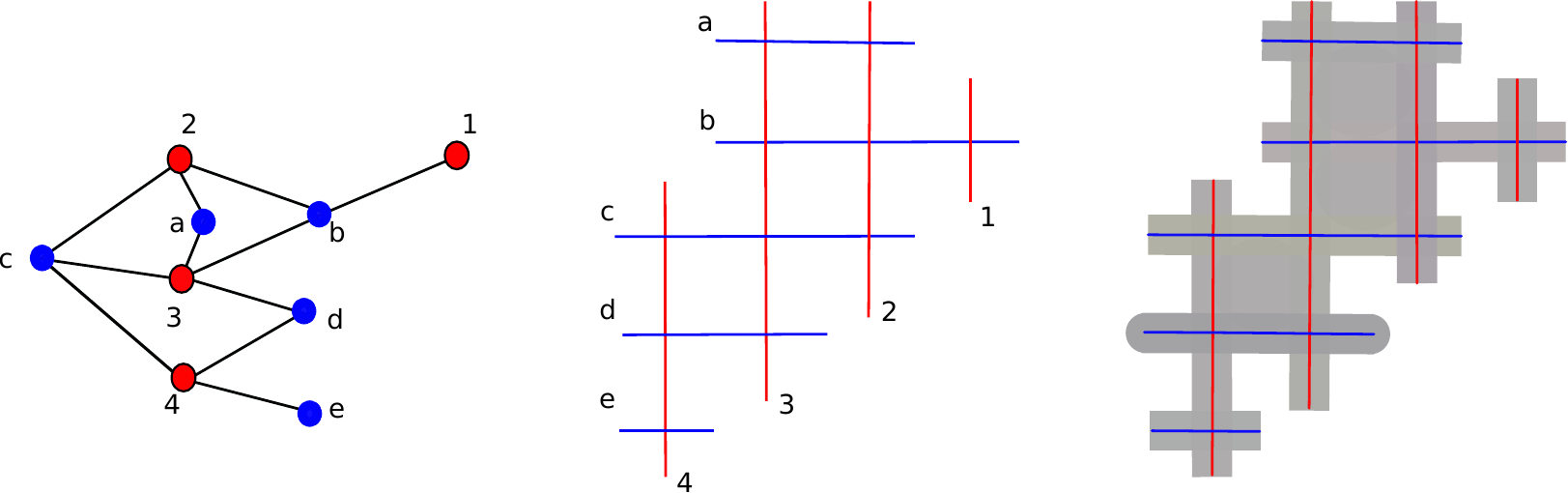} 
   \caption{}
   \label{lines1-fig}
\end{figure}
A planar realization is  {\it fillable} if it is possible to attach (possibly non-convex) polygons to
the planar graph along closed cycles, so that the interior of the polygon does not include any
endpoint of a line segment.  
Figure~\ref{lines1-fig} gives an example of a fillable planar realization, and Figure~\ref{lines2-fig} gives an example of a 
non-fillable planar realization.  

Think of the planar realization as being embedded in $S^3$.
Let $S$ be the  filled planar realization after gluing together each end of the horizontal
strips to its opposite with a single positive full twist, and the end of each vertical strip to its opposite by a single negative full twist.   
The boundary of $S$ is a link $K \subset S^3$ with distinguished Seifert surface $S$.
We call $(K,S)$ a {\it Coxeter link} associated to $\Gamma$.

\begin{figure}[htbp] 
   \centering
      \includegraphics[width=3.5in]{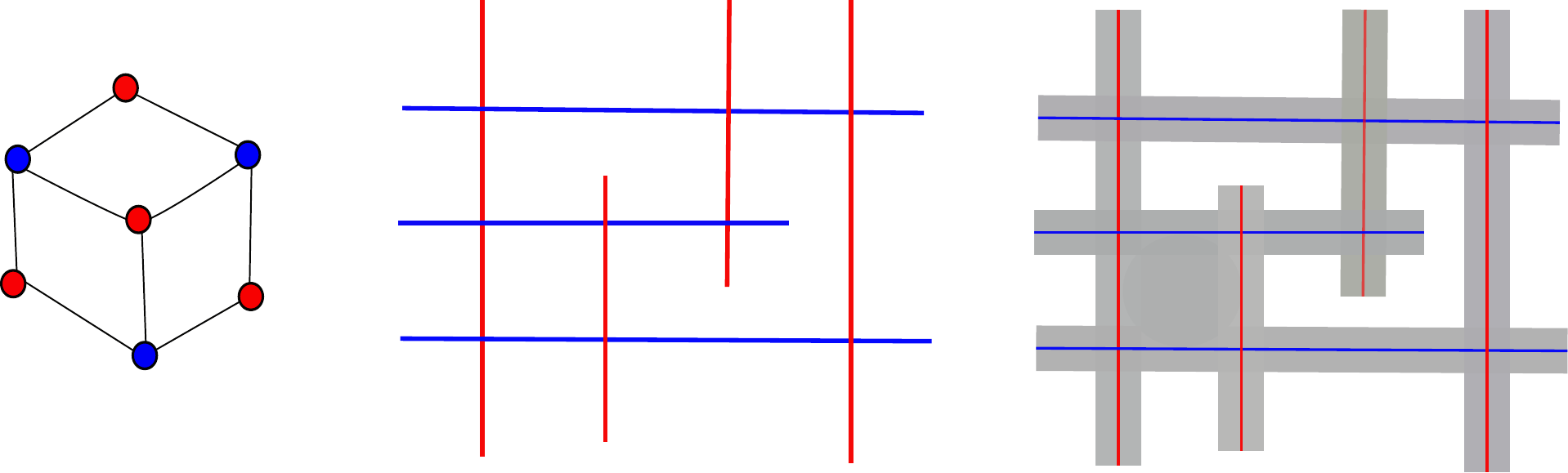} 
   \caption{}
   \label{lines2-fig}
\end{figure}

\begin{figure}[htbp] 
   \centering
   \includegraphics[width=2.5in]{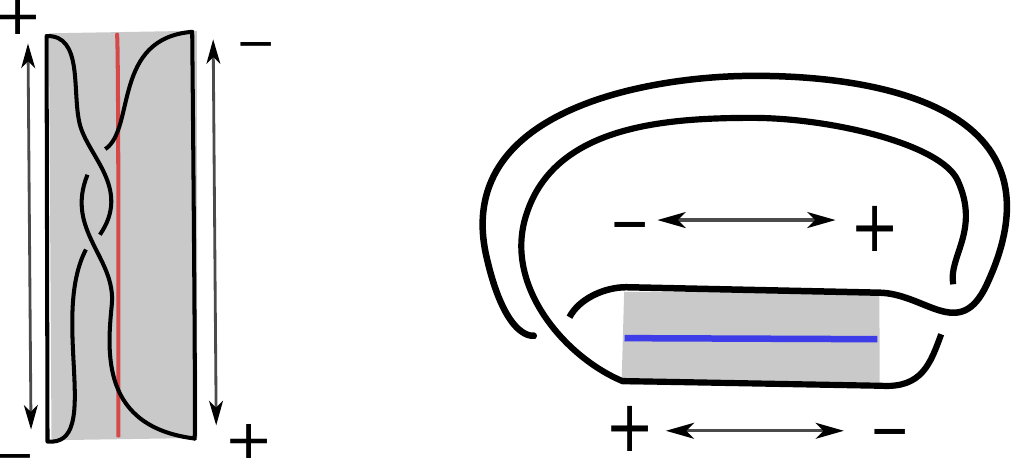} 
   \caption{}
   \label{leftright-fig}
\end{figure}

\begin{prop} If $\Gamma$ is an alternating-sign Coxeter graph with a fillable planar realization, then 
any associated Coxeter link is alternating.
\end{prop}

\begin{proof}
The link $K$ has an alternating planar diagram coming from drawing each vertical and horizontal Hopf band as in
Figure~\ref{leftright-fig}.  Here the shaded  rectangle is the original neighborhood of the line segment
associated to a vertex of $\Gamma$.  The signs indicate over~($+$) and under~($-$) crossings.  Thus we can see that
for each vertex $v \in V_\Gamma$, when proceeding along $\ell_v$ there is always a $-$ sign on the right
and a $+$ sign on the left, where $-$ indicates an upcoming underpass, and $+$ indicates an upcoming
overpass.  Since the signs are consistent on vertical and horizontal segments ($-$ appears on the right and
$+$ appears on the left no matter from which direction you approach an endpoint of a segment) the link $K$ is alternating.
\end{proof}

\begin{prop} \label{homdil-prop} The Coxeter link of an alternating-sign Coxeter graph is fibered, and the homological monodromy is
conjugate to $-C_{+-}$.  
\end{prop}

\begin{proof}
Since the surface $S$ can be obtained
from a disk by Hopf plumbings, the boundary of $S$ is a  fibered  link $K$ with fiber $S$.  All the
strips become annuli on $S$. The monodromy of the fibration is the product of right or left Dehn twists around core
curves of the annuli, right or left being determined by whether the twist is positive or negative
\cite{Gab, Mur, Sta}. 

Let $V_\Gamma$ be the set of vertices of $\Gamma$.  For $v \in V_\Gamma$, let
$\gamma_v$ be the closed curve defined by $\ell_v$.  Then the homology classes $[\gamma_v]$ form a basis for
$H_1(S;\mathbb R)$,  and the
 monodromy $\phi$ of
$S$ is the product of positive Dehn twists on  $\gamma_v$ 
for $v$ such that $\mathfrak s (v) = 1$  composed
with the product of negative Dehn twists on $\gamma_v$ for $v$ such that $\mathfrak s(v) = -1$.  
Let $\R^{V_\Gamma}$ be the vector space of $\R$-labelings 
of the vertices. For $v \in V_\Gamma$, let $[v]$ be the corresponding element of $\R^{V_\Gamma}$ giving the label
$1$ on $v$ and $0$ on all other vertices of $\Gamma$.  There
is a commutative diagram
$$
\xymatrix{
&\R^{V_\Gamma}\ar[d]_{-C_{+-}} \ar[r] &H_1(S;\R) \ar[d]^{\phi_*}\\
&\R^{V_\Gamma} \ar[r] &H_1(S;\R)
}
$$
where the horizontal arrows taking $[v]$ to $[\gamma_v]$ are isomorphisms.

The Coxeter transformation decomposes as
$$
C_{+-}  = C_+ C_- = - M (M^T)^{-1},
$$
where $M = -C_+$, cf.~\cite{Ho}.
By construction, $M$ is also the Seifert matrix for $S$  in $S^3 \setminus K$ with respect
to the generators for homology given by the core curves of the attached Hopf bands.  
Thus 
$$
\phi_* = (M^T)^{-1}M,
$$
see e.g.~\cite{Rol}, and is conjugate to $-C_{+-}$.
\end{proof}

\begin{cor} 
\label{Alexcor}
The Alexander polynomial $\Delta(t)$ satisfies
$$
\Delta(t) = c(-t),
$$
where $c(t)$ is the characteristic polynomial of the Coxeter transformation $C_{+-}$ of $\Gamma$.
\end{cor}

\begin{proof} 
The Alexander polynomial $\Delta_{S}(t)$ is the characteristic polynomial of $M(M^T)^{-1} = - C_{+-}$.
\end{proof}

\begin{ex}{\em
Figure~\ref{Coxeterlink-fig} gives an example of an alternating-sign Coxeter graph and 
fillable planar realization.

\begin{figure}[htbp] 
   \centering
   \includegraphics[width=4in]{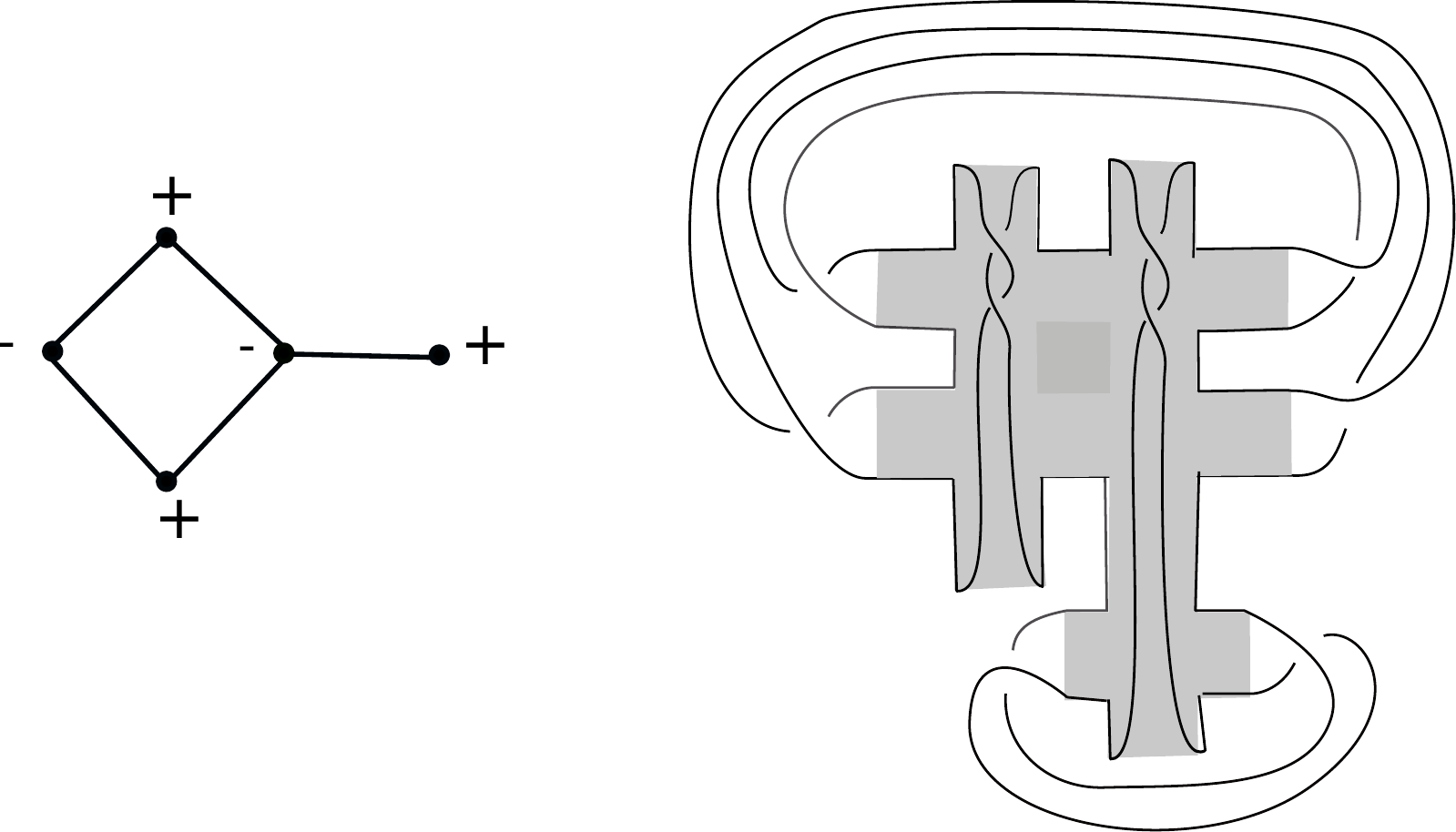} 
   \caption{}
   \label{Coxeterlink-fig}
\end{figure}

Then
$$
C_+ = 
\left  [
\begin{array}{rrrrr}
-1 & 0 & 0 & 1 & 1\\
0 & -1 & 0 & 1 & 1\\
0 & 0 & -1 & 0 & 1\\
0  & 0 & 0 & 1 & 0\\
0 & 0 & 0 & 0 & 1
\end{array}
\right ]
\qquad
C_- = 
\left  [
\begin{array}{rrrrr}
1 & 0 & 0 & 0 & 0\\
0 & 1 & 0 & 0 & 0\\
0 & 0 & 1 & 0 & 0\\
-1  & -1 & 0 & -1 & 0\\
-1 & -1 & -1 & 0 & -1
\end{array}
\right ].
$$
Setting the orientation on  the Seifert surface $S$ so that the shaded area is oriented positively toward the viewer,
we see that $-C_+$ is the Seifert matrix, and
$$
C_- =  - (C_+^T)^{-1}.
$$
The Coxeter transformation is given by
$$
C_{+-} = C_+C_- =
-\left  [
\begin{array}{rrrrr}
3 & 2 & 1 & 1 & 1\\
2 & 3 & 1 & 1 & 1\\
1 & 1 & 2 & 0 & 1\\
1  & 1 & 0 & 1 & 0\\
1 & 1 & 1 & 0 & 1
\end{array}
\right ].
$$
The associated Alexander and Coxeter polynomials are:
\begin{eqnarray*}
\Delta(t) &=& t^5 - 10t^4 + 27t^3 - 27 t^2 + 10 t -1\\
c(t) &=& t^5 + 10t^4 + 27 t^3 + 27 t^2 + 10t + 1.
\end{eqnarray*}
}
\end{ex}

\begin{rem}{\em The link associated to a Coxeter graph is not uniquely determined by the combinatorics of the graph. 
Figure~\ref{tree2-fig} shows two different planar embeddings of a Coxeter graph. The two links realizing these embeddings are distinct:
one of them has an unknotted component, while the other does not. While for a large class of classical Coxeter trees, two different planar embeddings
always yield distinct but mutant links by a theorem of Gerber \cite{Ge}, we do not know whether the same holds in the alternating-sign case.}
\end{rem}

\begin{figure}[htbp] 
   \centering
   \includegraphics[width=4in]{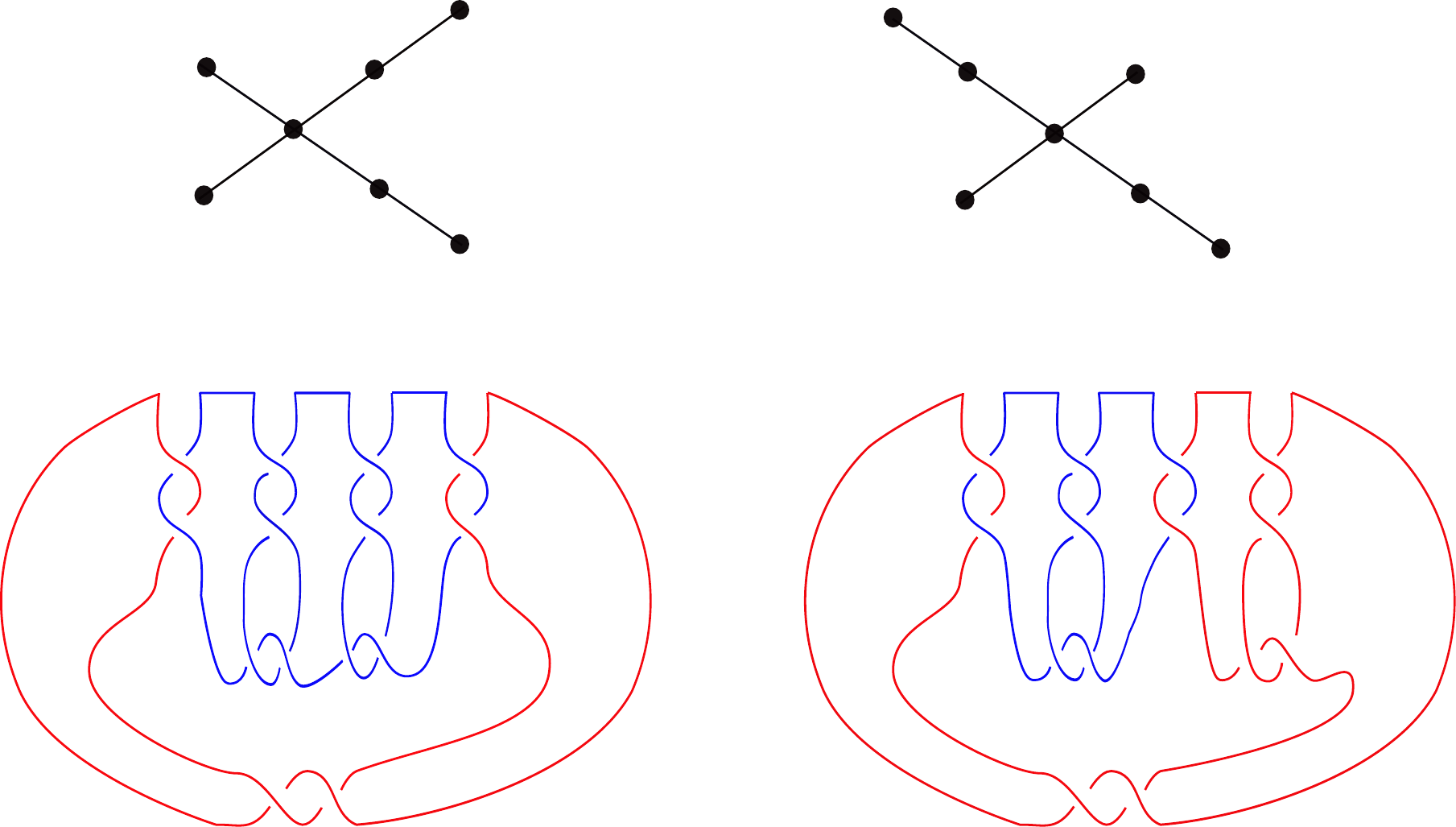} 
   \caption{}
   \label{tree2-fig}
\end{figure}

In general, even if $\Gamma$ does not have a planar realization, it is possible to find a  surface $S$ and a system of simple closed curves
 $\{\gamma_v\}$   in one-to-one correspondence with $V_\Gamma$
such that
\begin{enumerate}
\item the  intersection matrix of the $\gamma_v$ equals the adjacency matrix for $V_\Gamma$; and
\item the complementary components of the union of $\gamma_v$ are either disks or boundary parallel
annuli
\end{enumerate}
(see, e.g.~\cite{Hi2}).
Since $\Gamma$ is bipartite, the system of curves partitions into two multi-curves $\gamma_+$ and $\gamma_-$
that intersect transversally.   Let $\tau_+$ and $\tau_-$ be the positive Dehn twist along $\gamma_+$, respectively,
the negative Dehn twist along $\gamma_-$.   Let $\phi = \tau_+ \tau_-$.   We call $(S,\phi)$ a {\it geometric realization}
of $(\Gamma,\mathfrak s)$.

\begin{lem}\label{roots-lem} Let $E$ be the set of eigenvalues of $-C_{+-}$ and let $F$ be the set of eigenvalues of the
homological action of $\phi$.  Then 
$$
F \setminus \{1\} \subset E \setminus \{1\}.
$$
\end{lem}

\begin{proof} The proof follows along the same lines as the proof of Proposition~\ref{homdil-prop}, the only difference
being that the horizontal arrows in the commutative diagram need not be one-to-one or onto. The cokernel is generated
by boundary parallel curves whose homology classes are fixed  by $\phi_*$, hence their homology classes are contained
in the eigenspace for $1$.
\end{proof}

Let $(S,\phi)$ be a geometric realization of an alternating-sign Coxeter graph $(\Gamma,\mathfrak s)$.  Then
the eigenvalues of the homological action of $\phi$ are real and strictly positive by Proposition~\ref{bicolour-prop} and Lemma~\ref{roots-lem}.  
This implies Theorem~\ref{realroots-thm}. Similarly, Theorem~\ref{goldenmean-thm} follows directly from  Proposition~\ref{minimum-prop} and Lemma~\ref{roots-lem}.  

Combining Proposition~\ref{interlacing-prop} with Corollary~\ref{Alexcor}, we also have the following interlacing result.

\begin{thm}\label{interlacing-thm}  If $K'$ and $K$ are alternating-sign Coxeter links associated to $\Gamma'$ and $\Gamma$,
respectively, where $\Gamma'$ is a vertex extension of $\Gamma$, then the roots of the 
Alexander polynomial of $K'$ and that of $K$ are interlacing.
\end{thm}

%
%
%

\end{document}

%% file: non-interlacing.pdf_tex
\begingroup%
  \makeatletter%
  \providecommand\color[2][]{%
    \errmessage{(Inkscape) Color is used for the text in Inkscape, but the package 'color.sty' is not loaded}%
    \renewcommand\color[2][]{}%
  }%
  \providecommand\transparent[1]{%
    \errmessage{(Inkscape) Transparency is used (non-zero) for the text in Inkscape, but the package 'transparent.sty' is not loaded}%
    \renewcommand\transparent[1]{}%
  }%
  \providecommand\rotatebox[2]{#2}%
  \ifx\svgwidth\undefined%
    \setlength{\unitlength}{473.96546637bp}%
    \ifx\svgscale\undefined%
      \relax%
    \else%
      \setlength{\unitlength}{\unitlength * \real{\svgscale}}%
    \fi%
  \else%
    \setlength{\unitlength}{\svgwidth}%
  \fi%
  \global\let\svgwidth\undefined%
  \global\let\svgscale\undefined%
  \makeatother%
  \begin{picture}(1,0.19860337)%
    \put(0,0){\includegraphics[width=\unitlength]{non-interlacing.pdf}}%
    \put(0.0142469,0.10488329){\color[rgb]{0,0,0}\makebox(0,0)[lb]{\smash{$\Gamma$}}}%
    \put(0.90287234,0.09976817){\color[rgb]{0,0,0}\makebox(0,0)[lb]{\smash{$\Gamma'$}}}%
  \end{picture}%
\endgroup%